\documentclass[12pt]{amsart}
\usepackage{a4wide,enumerate}
\usepackage[english]{babel}
\usepackage{a4wide,color,graphicx,amsmath,amssymb,verbatim,mathrsfs,latexsym}
\allowdisplaybreaks

\let\pa\partial  
\let\na\nabla  
\let\eps\varepsilon  
  
\newcommand{\R}{{\mathbb R}} 
\newcommand{\diver}{\operatorname{div}}

\newtheorem{theorem}{Theorem}   
\newtheorem{lemma}[theorem]{Lemma}   
   
\newtheorem{remark}[theorem]{Remark}   
\newtheorem{corollary}[theorem]{Corollary}

\begin{document}  

\title[Cross-diffusion systems from population dynamics]{Qualitative behavior 
of solutions to cross-diffusion systems from population dynamics}

\author{Ansgar J\"ungel}
\address{Institute for Analysis and Scientific Computing, Vienna University of  
	Technology, Wiedner Hauptstra\ss e 8--10, 1040 Wien, Austria}
\email{juengel@tuwien.ac.at} 

\author{Nicola Zamponi}
\address{Institute for Analysis and Scientific Computing, Vienna University of  
	Technology, Wiedner Hauptstra\ss e 8--10, 1040 Wien, Austria}
\email{nicola.zamponi@tuwien.ac.at}

\date{\today}

\thanks{The authors acknowledge partial support from   
the Austrian Science Fund (FWF), grants P22108, P24304, and W1245} 

\begin{abstract}
A general class of cross-diffusion systems for two population species 
in a bounded domain with no-flux boundary conditions and Lotka-Volterra-type
source terms is analyzed. 
Although the diffusion coefficients are assumed to depend linearly on the 
population densities, the equations are strongly coupled.
Generally, the diffusion matrix is neither symmetric
nor positive definite. Three main results are proved: the existence of
global uniformly bounded weak solutions, their convergence to the constant steady
state in the weak competition case, and the uniqueness of weak solutions.
The results hold under appropriate conditions on the diffusion parameters
which are made explicit and which contain simplified Shigesada-Kawasaki-Teramoto
population models as a special case. The proofs are based on entropy methods,
which rely on convexity properties of suitable Lyapunov functionals.
\end{abstract}

% \paragraph{Keywords:}  
 
% \paragraph{AMS classification:}  

\maketitle

%%%%%%%%%%%%%%%%%%%%%%%%%%%%%%%%%%%%%%%%%%%%%%%%%%%%%%%%%%%%%%%%%%%%%%%%%%%%%%%

\section{Introduction}\label{sec.intro}

Many multi-species systems in biology, chemistry, and physics can be described
by reaction-diffusion systems with cross-diffusion effects. The analysis of such
problems is challenging since generally neither maximum principles nor
regularity theory can be applied. Moreover, many systems have diffusion 
matrices that are neither symmetric nor positive definite such that even the 
local-in-time existence of solutions is a nontrivial task. In this paper,
we apply and extend the boundedness-by-entropy method of \cite{Jue15} to a class of
cross-diffusion systems for two species, which are motivated from population 
dynamics. Compared to our previous work \cite{Jue15},
we are here interested in the qualitative behavior of weak solutions, namely
their uniform boundedness, positivity, large-time asymptotics, and uniqueness.

\subsection{Setting}

We consider reaction-diffusion systems of the form
\begin{equation}\label{1.eq}
  \pa_t u - \diver(A(u)\na u) = f(u)\quad\mbox{in }\Omega,\ t>0,
\end{equation}
subject to the homogeneous Neumann boundary and initial conditions
\begin{equation}\label{1.bic}
	(A(u)\na u)\cdot\nu=0\quad\mbox{on }\pa\Omega, \quad u(0)=u^0\quad\mbox{in }\Omega,
\end{equation}
where $u=(u_1,u_2)^\top$ represents the vector of the densities of the species, 
$A(u)=(A_{ij}(u))\in\R^{2\times 2}$ is the diffusion matrix,
and the birth-death processes are modeled by the function $f=(f_1,f_2)$.
Furthermore, $\Omega\subset\R^d$ ($d\ge 1$) is a bounded domain
with Lipschitz boundary and $\nu$ is the exterior unit normal vector to
$\pa\Omega$. Our main assumption is that the diffusivities
depend linearly on the densities,
\begin{equation}\label{1.Aij}
  A_{ij}(u) = \alpha_{ij} + \beta_{ij}u_1 + \gamma_{ij}u_2. \quad i,j=1,2,
\end{equation}
where $\alpha_{ij}$, $\beta_{ij}$, $\gamma_{ij}$ are real numbers.

Such models can be formally derived from a master equation for a random walk
on a lattice in the diffusion limit with transition rates which depend
linearly on the species' densities \cite[Appendix B]{Jue15}. 
They can be also deduced as the limit equations of an interacting particle system
modeled by stochastic differential equations with interaction forces
which depend linearly on the corresponding stochastic processes
\cite{GaSe14,Oel89}. 

The most prominent example of \eqref{1.Aij} is probably the 
population model of Shigesada, Kawasaki, and Teramoto \cite{SKT79}
(abbreviated SKT model):
\begin{equation}\label{1.SKT}
  A(u) = \begin{pmatrix}
	a_{10} + 2a_{11}u_1 + a_{12}u_2 & a_{12}u_1 \\
	a_{21}u_2 & a_{20} + a_{21}u_1 + 2a_{22}u_2
	\end{pmatrix},
\end{equation}
where the coefficients $a_{ij}$ are nonnegative, and 
the source terms in \eqref{1.eq} are given by
\begin{equation}\label{1.LV}
  f_i(u) = (b_{i0} - b_{i1}u_1 - b_{i2}u_2)u_i, \quad i=1,2,
\end{equation}
and the coefficients $b_{ij}$ are nonnegative.
The existence of global weak solutions without any restriction on the
diffusivities (except positivity) was achieved in 
\cite{GGJ03} in one space dimension and in \cite{ChJu04,ChJu06} 
in several space dimensions. Global classical solutions for constant $A_{ij}$
were shown to exist in \cite{LeNg15}.
Galiano \cite{Gal12} proved the uniqueness of bounded weak solutions to the 
SKT model with either diagonal diffusion matrix or the regularity
assumption $\na u_i\in L^\infty$. Uniqueness of strong solutions was 
shown by Amann \cite{Ama89} in the triangular case ($a_{21}=0$ in \eqref{1.SKT}).

There are much less results in the literature
concerning $L^\infty$ bounds and large-time asymptotics.
In one space dimension and with coefficients $a_{10}=a_{20}$, 
Shim \cite{Shi02} proved uniform upper bounds. Moreover, if cross-diffusion 
is weaker than self-diffusion (i.e.\ $a_{12}<a_{22}$, 
$a_{21}<a_{11}$), weak solutions are bounded
and H\"older continuous \cite{Le06}. The existence of global bounded solutions
in the triangular case (i.e.\ $a_{21}=0$) was shown in \cite{CLY03}.
In the triangular case, Le \cite{Le02} proved the existence
of a global attractor. With vanishing birth-death terms, it was shown
in \cite{ChJu06} that the solution to the SKT model converges exponentially
fast to the constant steady state. 

It cannot be expected that such results hold for any choice of the
parameters appearing in \eqref{1.Aij} and \eqref{1.LV}. 
For instance, system \eqref{1.eq} with
$$
  A(u) = \begin{pmatrix}
	1 & -u_1 \\ 0 & 1 
	\end{pmatrix}, \quad
	f(u) = \begin{pmatrix} 0 \\ u_1-u_2 \end{pmatrix}
$$
corresponds to the parabolic-parabolic Keller-Segel model which
exhibits the phenomenon of cell aggregation. If the cell density
is sufficiently large initially, finite-time
$L^\infty$ blow-up of solutions in two and three 
space dimensions occurs (see, e.g., \cite{HiPa09}), and bounded weak solutions
cannot be generally expected.

We wish to determine conditions on the parameters in \eqref{1.Aij} for which
the weak solutions to \eqref{1.eq}-\eqref{1.bic} are uniformly bounded, positive,
converge to the steady state, and are unique.
The key idea is to apply and refine entropy methods. 
Here, an entropy is a convex Lyapunov functional which provides 
additional gradient estimates. Special entropies may also allow for uniform
$L^\infty$ bounds, see below. The advantage of these methods is a separation
of the analytical and algebraic properties of the parabolic system.
Often, it is sufficient to analyze the algebraic structure of the
diffusion matrix, which simplifies the proofs, while achieving new results.

%%%%%%%%%%%

\subsection{Main results}

We introduce the triangle 
\begin{equation}\label{1.D}
  D = \{(u_1,u_2)\in\R^2:u_1>0,\ u_2>0,\ u_1+u_2<1\}.
\end{equation}
First, we prove the existence of global
bounded weak solutions to \eqref{1.eq}-\eqref{1.Aij} for diffusion matrices
of the form
$$
  A(u) = \begin{pmatrix}
	\alpha_{11} + \beta_{11}u_1 + \gamma_{11}u_2 & \beta_{12}u_1 \\
	\gamma_{21}u_2 & \alpha_{22} + \beta_{22}u_1 + \gamma_{22}u_2
	\end{pmatrix}.
$$

\begin{theorem}[Bounded weak solutions to \eqref{1.eq}]\label{thm.ex}
Let $u^0=(u_1^0,u_2^0)\in L^1(\Omega;\R^2)$ be such that $u^0(x)\in D$ 
for $x\in\Omega$, let $A(u)$ be given
by \eqref{1.Aij} with coefficients satisfying
\begin{align}
  & \alpha_{12} = \alpha_{21} = \beta_{21} = \gamma_{12} = 0, \label{1.symm1} \\
	& \beta_{22} = \beta_{11}-\gamma_{21}, \quad
	\gamma_{11} = \gamma_{22}-\beta_{12}, \quad 
	\gamma_{21} = \alpha_{22} - \alpha_{11} + \beta_{12}, \label{1.symm2} \\
	& \alpha_{11}>0,\quad \alpha_{22}>0, \quad 
	\beta_{12}<\alpha_{11}+\min\{\beta_{11},\gamma_{22}\}, \quad
	\alpha_{11}+\beta_{11} \ge 0, \quad \alpha_{22}+\gamma_{22} \ge 0, \label{1.cond}
\end{align}
and let $f_i(u)=u_ig_i(u)$, where $g_i(u)$ is continuous in $\overline{D}$
and nonpositive in $\{1-\eps < u_1+u_2 < 1\}$ for some $\eps>0$ $(i=1,2)$.
Then there exists a bounded nonnegative weak solution $u=(u_1,u_2)$ to 
\eqref{1.eq}-\eqref{1.bic} satisfying $u(x,t)\in \overline{D}$ for $x\in\Omega$, $t>0$,
\begin{equation}\label{1.reg}
  u\in L^2_{\rm loc}(0,\infty;H^1(\Omega;\R^2)), \quad
	\pa_t u\in L^2_{\rm loc}(0,\infty;H^1(\Omega;\R^2)'),
\end{equation}
and the initial datum is satisfied in the sense of $L^2$.
\end{theorem}

Note that the $L^\infty$ bound on $u$ is uniform in time.
We show in the appendix that \eqref{1.symm1} (and two further conditions)
are necessary to apply the entropy method. Thus, in the framework of such
techniques, conditions \eqref{1.symm1} cannot be improved.
The theorem also holds true if $\alpha_{11}=\alpha_{22}=0$ but
$\beta_{11}>0$ and $\gamma_{22}>0$; see Remark \ref{rem}.
The condition $u_1^0+u_2^0<1$ can be satisfied after a suitable scaling
of the positive function $u^0\in L^\infty(\Omega;\R^2)$ and is therefore not a 
restriction. The assumption on $f(u)$ guarantees that
the triangle $D$ is an invariant region under the action of the reaction terms. 
Theorem \ref{thm.ex} generalizes the global existence result in \cite{GaSe14},
where the positive definiteness of $A$ was needed. To the best of our knowledge,
this is the first general existence result for uniformly bounded weak solutions
to cross-diffusion systems with linear diffusivities.

The proof is based on the boundedness-by-entropy
method, first used in \cite{BDPS10} for a ion-transport model and later
extended in \cite{Jue15}.
The key idea is to formulate conditions under which the functional
$$
  \mathcal H[u] = \int_\Omega h(u)dx, \quad\mbox{where }
	h(u) = \sum_{i=1}^3 u_i(\log u_i-1), \ u_3:=1-u_1-u_2,
$$
is an entropy for \eqref{1.eq}. More precisely,
assume that the derivative of the entropy density $h:D\to\R$ is invertible
and the matrix $h''(u)A(u)$ is positive semidefinite, where
$h''(u)$ is the Hessian of $h(u)$.
We introduce the entropy variable $w=h'(u)$. Then \eqref{1.eq} is equivalent to
$$
  \pa_t u - \diver(B(w)\na w) = f(u(w)),
$$
where $B(w)=A(u)h''(u)^{-1}$ and $u(w)=(h')^{-1}(w)$. Now, if $f(u)\cdot w\le 0$,
$$
  \frac{d}{dt}{\mathcal H}[u] \le -\int_\Omega \na w:B(w)\na w dx
	= -\int_\Omega \na u:h''(u)A(u)\na u dx \le 0,
$$
where ``:'' denotes summation over both matrix indices. This shows that 
$\mathcal H[u]$ is a Lyapunov functional for \eqref{1.eq}. There is a second
consequence: Since the triangle $D$ in \eqref{1.D} is bounded, the original
variable $u=(h')^{-1}(w)$ maps into $D$ which is bounded. Therefore, 
$u(x,t)\in D$ and the solutions to \eqref{1.eq} are bounded. This result
holds without the use of a maximum principle.

Theorem \ref{thm.ex} can be applied to the SKT model \eqref{1.SKT}
to determine conditions under which this model possesses bounded weak 
solutions; see Section \ref{sec.skt}.
The novelty is not the global existence
(which has been proven in \cite{ChJu04})
but the uniform boundedness of weak solutions. 
	
The second main result is concerned with the large-time behavior of the
solutions to \eqref{1.eq}. The steady state of \eqref{1.eq}-\eqref{1.bic} 
is defined as the only constant solution $U=(U_1,U_2)$ to \eqref{1.eq}-\eqref{1.bic}.
(There may be also non-constant steady states \cite{LoNi96} but we are interested
only in constant solutions.) The steady state is a solution to the algebraic
equation $f(U)=0$. If $f$ is given by \eqref{1.LV} and $(b_{ij})_{i,j=1,2}$
is positive definite, equation $f(U)=0$ admits the unique solution
\begin{equation}\label{1.U}
  U_1 = \frac{b_{10}b_{22}-b_{20}b_{12}}{b_{11}b_{22}-b_{12}b_{21}}, \quad
	U_2 = \frac{b_{20}b_{11}-b_{10}b_{21}}{b_{11}b_{22}-b_{12}b_{21}}.
\end{equation}

\begin{theorem}[Convergence to the steady state]\label{thm.time}
Let the hypotheses of Theorem \ref{thm.ex} hold and let $f(u)$ be given by
\eqref{1.LV}. Let the matrix $(b_{ij})_{i,j=1,2}$ be positive definite and
assume that
\begin{equation}\label{1.bij}
  b_{10} = b_{12} < b_{11}, \quad b_{20} = b_{21} < b_{22},
\end{equation}
as well as
\begin{equation}\label{1.abc}
\begin{aligned}
  (\alpha_{11}+\beta_{11})(\alpha_{11}+\beta_{11}-\beta_{12})
	 - 4\gamma_{21}^2\frac{U_2}{U_1} &> 0, \\
  (\alpha_{22}+\gamma_{22})(\alpha_{22}+\gamma_{22}-\gamma_{21})
	 - 4\beta_{12}^2\frac{U_1}{U_2} &> 0.
\end{aligned}
\end{equation}
Then the solution to \eqref{1.eq}-\eqref{1.Aij} constructed in Theorem \ref{thm.ex}
satisfies $u_i(x,t)>0$ a.e.\ in $\Omega\times(0,\infty)$,
$u_i-U_i$, $\na\log u_i\in L^2(\Omega\times(0,\infty))$, and 
$$
  u_i(t)\to U_i\quad\mbox{strongly in }L^2(\Omega)\mbox{ as }t\to\infty,
	\quad i=1,2.
$$
\end{theorem}

Assumption \eqref{1.bij} is a special case of the weak competition case,
\begin{equation}\label{1.wcc}
  \frac{b_{11}}{b_{21}} > \frac{b_{10}}{b_{20}} > \frac{b_{12}}{b_{22}},
\end{equation}
which allows for coexistence of species in the Lotka-Volterra differential
equations \cite{Bro80}. This condition guarantees that $U\in D$, i.e.,
$U_1$, $U_2>0$ and $U_1+U_2<1$. 
The idea of the proof is to show that under the stated conditions
on the parameters, the functional
$$
  \Phi(u|U) = \sum_{i=1}^2\int_\Omega\bigg(u_i-U_i+U_i\log\frac{u_i}{U_i}\bigg)dx
$$
is a Lyapunov functional and satisfies
\begin{equation}\label{1.Phi}
  \frac{d}{dt}\Phi(u(t)|U) + c_b\int_0^t\|u(s)-U\|_{L^2(\Omega)}^2 ds
	+ c\sum_{i=1}^2\int_0^t\int_\Omega|\na\log u_i|^2 dxds \le 0,
\end{equation}
where $c_b>0$ is the smallest (positive) eigenvalue of $(b_{ij})_{i,j=1,2}$
and $c>0$ is another constant. For this property, we need condition \eqref{1.abc}.
Clearly, \eqref{1.Phi} is
only formal as $u_i$ may vanish, and we need to regularize to make this
inequality rigorous (see Section \ref{sec.time}). Inequality \eqref{1.Phi}
is the key step to deduce the properties mentioned in the theorem.

Our final result is the uniqueness of weak solutions to \eqref{1.eq}.

\begin{theorem}[Uniqueness of weak solutions]\label{thm.uni}
Let the assumptions of Theorem \ref{thm.ex} hold. Furthermore, let $f=0$ and
\begin{equation}\label{1.condu}
  \alpha_{22}=\alpha_{11}, \quad \gamma_{21}=\beta_{12}, 
	\quad \gamma_{22} = \beta_{11}.
\end{equation}
Then the weak solution to \eqref{1.eq}-\eqref{1.Aij} is unique.
\end{theorem}

Summarizing the assumptions on the parameters, the uniqueness result holds
for diffusion matrices of the form 
$$
  A(u) = \begin{pmatrix}
	\alpha_{11} + \beta_{11}u_1 + (\beta_{11}-\beta_{12})u_2 & \beta_{12}u_1 \\
	\beta_{12}u_2 & \alpha_{11} + (\beta_{11}-\beta_{12})u_1 + \beta_{11}u_2
	\end{pmatrix}.
$$
For the proof of Theorem \ref{thm.uni}, we first observe that under the conditions
imposed on the parameters in \eqref{1.Aij}, the sum $\rho:=u_1+u_2$
satisfies the diffusion equation $\pa_t\rho=\Delta F(\rho)$ for a certain 
nondecreasing function $F$. By the $H^{-1}$ method, this equation is uniquely
solvable. Furthermore, the difference $\sigma:=u_1-u_2$ solves the
drift-diffusion equation $\pa_t\sigma=\diver(d(\rho)\na\sigma+\sigma\na V(\rho))$
for certain functions $d(\rho)>0$ and $V(\rho)$. To prove the uniqueness of
{\em weak} solutions to this equation, we employ the method of Gajewski
\cite{Gaj94}. We stress the fact that we require only the regularity
$V(\rho)\in L^2(0,T;H^1(\Omega))$, which excludes many uniqueness techniques.
The idea of Gajewski is to differentiate the semimetric
$$
  \Xi[\sigma_1,\sigma_2] = S[\rho_1]+S[\rho_2]
	-2S\bigg[\frac{\sigma_1+\sigma_2}{2}\bigg], \quad\mbox{where }
	S[\sigma]=\int_\Omega \sigma\log\sigma dx,
$$
for two solutions $\sigma_1$ and $\sigma_2$ with respect to time and to show
that $\pa_t\Xi[\sigma_1(t),\sigma_2(t)]\le 0$ for $t>0$. Since 
$\Xi[\sigma_1(0),\sigma_2(0)]=0$, we infer from the nonnegativity of
$\Xi$ that $\Xi[\sigma_1(t),\sigma_2(t)]=0$ for all $t\ge 0$, and the
convexity of $\sigma\log\sigma$ shows that 
 $\sigma_1(t)=\sigma_2(t)=0$ for $t\ge 0$.

This paper is organized as follows.
Theorems \ref{thm.ex}, \ref{thm.time}, \ref{thm.uni} are proved in, respectively, 
Sections \ref{sec.thm}, \ref{sec.time}, \ref{sec.uni}. In the Appendix, we
derive some necessary conditions on the parameters in \eqref{1.Aij} to
apply the entropy method.

%%%%%%%%%%%%%%%%%%%%%%%%%%%%%%%%%%%%%%%%%%%%%%%%%%%%%%%%%%%%%%%%%%%%%%%%%%%%%%%

\section{Proof of Theorem \ref{thm.ex}}\label{sec.thm}

We apply the following theorem from \cite[Theorem 2]{Jue15}, here in a formulation
which is adapted to our situation.

\begin{theorem}[\cite{Jue15}]\label{thm.Jue15}
Let $D\subset(0,1)^2$ be a bounded domain, $u^0\in L^1(\Omega;\R^2)$ with 
$u^0(x)\in D$ for $x\in\Omega$ and assume that
\begin{description}
\item[\rm H1] There exists a convex function $h\in C^2(D;[0,\infty))$ 
such that its derivative $h':D\to\R^n$ is invertible.
\item[\rm H2] Let $\alpha^*>0$, $0\le m_i\le 1$ $(i=1,2)$ be such that for all 
$z=(z_1,z_2)^\top\in\R^2$ and $u=(u_1,u_2)^\top\in D$,
$$
  z^\top h''(u)A(u)z \ge \alpha^*\sum_{i=1}^2 u_i^{2(m_i-1)}z_i^2.
$$
\item[\rm H3] It holds $A\in C^0(D;\R^{2\times 2})$, $f\in C^0(D;\R^2)$,
and there exists $c_f>0$ 
such that for all $u\in D$, $f(u)\cdot h'(u)\le c_f(1+h(u))$.
\end{description}
Then there exists a weak solution $u$ to \eqref{1.eq}-\eqref{1.bic}
satisfying $u(x,t)\in\overline{D}$ for $x\in\Omega$, $t>0$ and
\begin{equation*}
  u\in L^2_{\rm loc}(0,\infty;H^1(\Omega;\R^2)), 
	\quad\pa_t u \in L^2_{\rm loc}(0,\infty;H^1(\Omega;\R^2)').
\end{equation*}
The initial datum is satisfied in the sense of $L^2$.
Moreover, if $h\in C^0(\overline D)$ and $f(u)\cdot h'(u)\le 0$ for all $u\in D$, 
the entropy $\mathcal H[u(\cdot,t)]=\int_\Omega h(u(x,t))dx$ is nonincreasing
in time.
\end{theorem}

The last statement is a consequence of the proof of the theorem in \cite{Jue15}.

Now, choose the entropy density
\begin{equation}\label{a.h}
  h(u|\overline u) = \sum_{i=1}^3\overline u_i\bigg(\frac{u_i}{\overline u_i}
	\log\frac{u_i}{\overline u_i} - \frac{u_i}{\overline u_i} + 1\bigg),
	\quad u_3 = 1-u_1-u_2, \quad
	\overline u_3 = 1-\overline u_1-\overline u_2,
\end{equation}
defined on $D$ (see \eqref{1.D}). This function fulfills Hypothesis H1.
It remains to verify Hypotheses H2 and H3.

\subsection{Verification of Hypothesis H2}

Let $H(u)=h''(u)$. We require that the matrix $H(u)A(u)$ is symmetric. 
This leads to conditions \eqref{1.symm1}-\eqref{1.symm2},
and we are left with the five parameters $\alpha_{11}$, $\alpha_{22}$, $\beta_{11}$,
$\beta_{12}$, and $\gamma_{22}$. We prove that $H(u)A(u)$ is positive definite
under additional assumptions.

\begin{lemma}\label{lem.HA}
Let conditions \eqref{1.symm1}-\eqref{1.cond} hold.
Then there exists $\eps>0$ such that for all $z\in\R^2$ and all $u\in D$,
\begin{equation}\label{a.H2}
  z^\top H(u)A(u)z \ge \eps\left(\frac{z_1^2}{u_1}+\frac{z_2^2}{u_2}\right).
\end{equation}
\end{lemma}

The lemma shows that Hypothesis H2 is fulfilled with $m_i=\frac12$.
First, we verify the following result.

\begin{lemma}\label{lem.HA2}
The matrix $H(u)A(u)$ is positive semidefinite for all $u\in D$ if and only if
\begin{equation}\label{a.cond}
  \alpha_{11}\ge 0,\quad \alpha_{22}\ge 0, \quad 
	\beta_{12}\le\alpha_{11}+\min\{\beta_{11},\gamma_{22}\}, \quad
	\alpha_{11}+\beta_{11}\ge 0, \quad \alpha_{22}+\gamma_{22}\ge 0.
\end{equation}
\end{lemma}

\begin{proof}
{\em Step 1: equations \eqref{a.cond} are necessary.}
We first prove that the positive semidefiniteness of $H(u)A(u)$ implies
\eqref{a.cond} by studying $H(u)A(u)$ close to the vertices of $D$. To this end,
we define the matrix-valued functions
\begin{align*}
  & F_1(s)=s H(s,s)A(s,s), \quad F_2(s) = sH(1-2s,s)A(1-2s,s), \\
	& F_3(s) = sH(s,1-2s)A(s,1-2s) \quad\mbox{for } s\in(0,\tfrac12).
\end{align*}
A straightforward computation shows that
\begin{align*}
  \lim_{s\to 0+}F_1(s) &= \begin{pmatrix} \alpha_{11} & 0 \\ 0 & \alpha_{22}
	\end{pmatrix}, \quad
  \lim_{s\to 0+}F_2(s) = \begin{pmatrix} 
	\alpha_{11}+\beta_{11} & \alpha_{11}+\beta_{11} \\
	\alpha_{11}+\beta_{11} & 2(\alpha_{11}+\beta_{11})-\beta_{12} \end{pmatrix}, \\
	\lim_{s\to 0+}F_3(s) &= \begin{pmatrix} 
	\alpha_{11}+\alpha_{22}+2\gamma_{22}-\beta_{12} & \alpha_{22}+\gamma_{22} \\
	\alpha_{22}+\gamma_{22} & \alpha_{22}+\gamma_{22}
	\end{pmatrix}.
\end{align*}
Since $H(u)A(u)$ is assumed to be positive semidefinite on $D$, also
$\lim_{s\to 0+}F_i(s)$ must be positive semidefinite for $i=1,2,3$.
Sylvester's criterion applied to these matrices yields \eqref{a.cond} since
\begin{align*}
  \det\big(\lim_{s\to 0+}F_2(s)\big) &= (\alpha_{11}+\beta_{11})
	(\alpha_{11}+\beta_{11}-\beta_{12}) \ge 0, \\
  \det\big(\lim_{s\to 0+}F_3(s)\big) &= (\alpha_{22}+\gamma_{22})
	(\alpha_{11}+\gamma_{22}-\beta_{12}) \ge 0.
\end{align*}

{\em Step 2: sign of the diagonal elements of $HA$.}
Let conditions \eqref{a.cond} hold.
We claim that either $HA:=H(u)A(u)$ is positive semidefinite or one of the
two coefficients $(HA)_{11}$ or $(HA)_{22}$ is positive in $D$.
For this, we introduce the functions
\begin{align*}
  f_1(u_2,u_3) &= (1-u_2-u_3)u_3(HA)_{11}(1-u_2-u_3,u_2), \quad (u_2,u_3)\in D, \\
	f_2(u_1,u_3) &= (1-u_1-u_3)u_3(HA)_{22}(u_1,1-u_1-u_3), \quad (u_1,u_3)\in D.
\end{align*}
We wish to apply the strong maximum principle to $f_1$ and $f_2$. In fact,
$f_1$ and $f_2$ are nonnegative on $\pa D$ since \eqref{a.cond} implies that
\begin{align}
  f_1|_{u_3=1-u_2} &= (1-u_2)\big(\alpha_{11}+(\gamma_{22}-\beta_{12})u_2\big)
	\ge \alpha_{11}(1-u_2)^2 \ge 0, \label{a.f11} \\
	f_1|_{u_2=0} &= \alpha_{11} + \beta_{11}(1-u_3) \ge \alpha_{11}u_3 \ge 0, 
	\label{a.f12} \\
	f_1|_{u_3=0} &= (1-u_2)\big((\alpha_{11}+\beta_{11})(1-u_2)
	+ \alpha_{22}+\gamma_{22}\big) \ge 0, \label{a.f13} \\
	f_2|_{u_1=0} &= \alpha_{22} + \gamma_{22}(1-u_2) \ge \alpha_{22}u_2 \ge 0, 
	\label{a.f21} \\
	f_2|_{u_3=1-u_1} &= (1-u_1)\big(\alpha_{22}(1-u_1) 
	+ (\alpha_{11}+\beta_{11}-\beta_{12})u_1\big) \ge \alpha_{22}(1-u_1)^2 \ge 0, 
	\label{a.f22} \\
	f_2|_{u_3=0} &= (1-u_1)\big((\alpha_{22}+\gamma_{22})(1-u_1)
	+ (\alpha_{11}+\beta_{11})u_1\big) \ge 0. \label{a.f23} 
\end{align}
Furthermore, a straightforward computation gives
$$
  \Delta_{(u_2,u_3)}f_1 = -\Delta_{(u_1,u_3)}f_2 
	= 2(\alpha_{11}-\alpha_{22}+\beta_{11}-\gamma_{22}) \quad\mbox{in }D.
$$
Consequently, either $\Delta_{(u_2,u_3)}f_1\le 0$ or $\Delta_{(u_1,u_3)}f_2\le 0$
in $D$. By the strong maximum principle, there exists $i\in\{1,2\}$ such that
$f_i>0$ in $D$ unless $f_i\equiv 0$ in $D$. This means that $(HA)_{ii}>0$ in $D$
unless $(HA)_{ii}\equiv 0$ in $D$.

To complete the claim, we show that if one of the coefficients $(HA)_{11}$ or
$(HA)_{22}$ is identically zero in $D$, then $HA$ is positive semidefinite in $D$.
Consider first the case $(HA)_{11}\equiv 0$ in $D$, i.e.\ $f_1\equiv 0$ in $D$.
Then also $f_1\equiv 0$ on $\pa D$. We deduce from \eqref{a.f11}-\eqref{a.f13}
the relations $\alpha_{11}=\beta_{11}=0$, $\alpha_{22}=-\gamma_{22}$,
and $\gamma_{22}=\beta_{12}$ and so, 
$$
  HA = \alpha_{22}\begin{pmatrix} 0 & 0 \\ 0 & 1/u_2 \end{pmatrix}.
$$
Since $\alpha_{22}\ge 0$, $HA$ is positive semidefinite. In the remaining case
$(HA)_{22}\equiv 0$ in $D$, \eqref{a.f21}-\eqref{a.f23} lead to
$$
  HA = \alpha_{11}\begin{pmatrix} 1/u_1 & 0 \\ 0 & 0 \end{pmatrix},
$$
and because of $\alpha_{11}\ge 0$, this matrix is positive semidefinite.
This shows the claim.

{\em Step 3: sign of the determinant of $HA$.}
By Step 2, we can assume that one of the two coefficients $(HA)_{11}$ or 
$(HA)_{22}$ is positive in $D$. We show that $\det A\ge 0$ in $D$. Then
$\det(HA)=\det H\det A\ge 0$ in $D$, and by Sylvester's criterion,
these properties give the positive semidefiniteness of $HA$. This proves
that conditions \eqref{a.cond} are sufficient for the positive semidefiniteness
of $HA$.

We consider $\det A$ on $\pa D$. Taking into account conditions \eqref{a.cond},
we find that
\begin{align*}
  \det A(0,u_2) &= (\alpha_{22}+\gamma_{22}u_2)\big(\alpha_{11}
	+(\gamma_{22}-\beta_{12})u_2\big) 
	\ge \alpha_{22}(1-u_2)\alpha_{11}(1-u_2) \ge 0, \\
	\det A(u_1,0) &= (\alpha_{11}+\beta_{11}u_1)\big(\alpha_{22}(1-u_1)
	+ (\alpha_{11}+\beta_{11}-\beta_{12})u_1\big) \\
	&\ge \alpha_{22}(1-u_1)\alpha_{11}(1-u_1) \ge 0, \\
	\det A(u_1,1-u_1) &= \big((\alpha_{22}+\gamma_{22})(1-u_1)+\alpha_{11}+\beta_{11}
	\big) \\
	&\phantom{xx}{}\times\big(\alpha_{11}-\beta_{12}+\gamma_{22}
	+(\beta_{11}-\gamma_{22})u_1\big) \\
	&\ge (\alpha_{11}+\beta_{11})\big(-\min\{\beta_{11}-\gamma_{22},0\}
	+(\beta_{11}-\gamma_{22})u_1\big) \ge 0.
\end{align*}
We conclude that $\det A\ge 0$ on $\pa D$. 

Next, we consider the Hessian $C:=(\det A)''(u)$ with respect to $u$. Since
$\det A$ is a (multivariate) quadratic polynomial in $u$, 
$C$ is a symmetric constant matrix satisfying
$$
  \det C = -\big(\beta_{11}\beta_{12}+\gamma_{22}(\alpha_{11}-\alpha_{22}-\beta_{12})
	\big)^2 \le 0.
$$
Thus, one of the two eigenvalues of $C$ is nonpositive, say $\lambda\le 0$.
Let $v\in\R^2\backslash\{0\}$ be a corresponding eigenvector, i.e.\ $Cv=\lambda v$.
Furthermore, let $u\in D$ be arbitrary and 
let $I_u\subset\R$ be the (unique) bounded open
interval containing zero with the property that the segment $u+I_uv$ is contained
in $D$ and its extreme points belong to $\pa D$. Define $\phi(r)=\det A(u+rv)$
for $r\in I_u$. Then $\phi''(r)=v^\top Cv=\lambda|v|^2\le 0$ for all $r\in I_u$.
We infer that $\phi$ is concave and attains its minimum at the border of $I_u$.
Since $\det A\ge 0$ on $\pa D$, this implies that $\det A(u+rv)\ge 0$
for all $r\in I_u$. By choosing $r=0\in I_u$, we conclude that $\det A(u)\ge 0$.
As $u\in D$ was arbitrary, this finishes the proof.
\end{proof}

\begin{proof}[Proof of Lemma \ref{lem.HA}]
The claim \eqref{a.H2} is equivalent to the positive semidefiniteness of the
matrix $HA-\eps\Lambda$ for a suitable $\eps>0$, where
$$
  \Lambda = \begin{pmatrix} 1/u_1 & 0 \\ 0 & 1/u_2 \end{pmatrix}.
$$
Since $\Lambda=HP$, where
$$
  P = \begin{pmatrix} 1-u_1 & -u_1 \\ -u_2 & 1-u_2 \end{pmatrix},
$$
we can write $HA-\eps\Lambda=HA^\eps$ with $A^\eps=A-\eps P$. We observe that
$A^\eps$ has the same structure as $A$ with the parameters
$$
  \alpha_{11}^\eps = \alpha_{11}-\eps, \quad \alpha_{22}^\eps=\alpha_{22}-\eps,
	\quad \beta_{11}^\eps = \beta_{11}+\eps, \quad 
	\beta_{12}^\eps = \beta_{12} + \eps, 
	\quad \gamma_{22}^\eps = \gamma_{22} + \eps.
$$
{}From Lemma \ref{lem.HA} we conclude that $HA^\eps$ is positive semidefinite if and
only if \eqref{a.cond} holds for the parameters 
$(\alpha_{11}^\eps,\alpha_{22}^\eps,\beta_{11}^\eps,\beta_{12}^\eps,\gamma_{22}^\eps)$ 
instead of $(\alpha_{11},\alpha_{22},\beta_{11},\beta_{12},\gamma_{22})$.
This means that $HA-\eps\Lambda$ is positive semidefinite for a suitable $\eps>0$
if and only if \eqref{1.cond} holds.
\end{proof}

\begin{remark}\label{rem}\rm
Let $\alpha_{11}=\alpha_{22}=0$ but $\beta_{11}>0$ and $\gamma_{22}>0$. 
We claim that there exists $\eps>0$ such that for all $z\in\R^2$ and $u\in D$,
$$
  z^\top H(u)A(u)z \ge \eps|z|^2
$$
holds, i.e., Hypothesis H2 is satisfied for $m_i=1$, and the conclusion of
Theorem \ref{thm.ex} holds. We show that
$HA-\eps{\mathbb I}$ is positive semidefinite, where ${\mathbb I}$ is the
identity matrix in $\R^{2\times 2}$. The matrix can be written as
$$
  HA-\eps{\mathbb I} = (HA)^\eps + \frac{\eps}{1-u_1-u_2}
	\begin{pmatrix} 1 & 1 \\ 1 & 1 \end{pmatrix},
$$
where $(HA)^\eps$ has the same structure as $HA$ but with
$\beta_{11}$, $\beta_{12}$, $\gamma_{22}$ replaced by 
$\beta_{11}^\eps=\beta_{11}-\eps$, $\beta_{12}^\eps=\beta_{12}-\eps$,
$\gamma_{22}^\eps=\gamma_{22}-\eps$.
Choosing $0<\eps\le\min\{\beta_{11},\gamma_{22}\}$, conditions \eqref{a.cond}
are satisfied for these parameters. Thus, Lemma \ref{lem.HA} shows that
$(HA)^\eps$ is positive semidefinite and we conclude that also
$HA-\eps{\mathbb I}$ is positive semidefinite, proving the claim.
\qed
\end{remark}

%%%%%%%%%%%%%%%%%%%%%%%%%%%

\subsection{Verification of H3}

By definition of $f_i$, we write
$$ 
  f_i(u)\pa_{u_i}h(u) = u_ig_i(u)\log u_i - u_ig_i(u)\log(1-u_1-u_2) 
	- u_i g_i(u)\log(\overline{u}_i/\overline{u}_3). 
$$
Since $g_i(u)$ and $u_i\log u_i$ are bounded in $\overline{D}$, the first
term on the right-hand side is bounded. The second term is bounded
in $\{0<u_1+u_2\le 1-\eps\}$ by a constant which depends on $\eps$.
Moreover, we have $g_i(u)\le 0$ in $\{1-\eps<u_1+u_2<1\}$ by assumption,
which implies that $-u_ig_i(u)\log(1-u_1-u_2)\le 0$ in $\{1-\eps<u_1+u_2<1\}$.
Finally, the third term is trivially bounded.
Thus, $f_i(u)\pa_{u_i}(u)\le c$ for a suitable constant $c>0$.

%%%%%%%%%%%%%%%%%%%%%%

\subsection{Bounded weak solutions to the SKT model}\label{sec.skt}

Applying Theorem \ref{thm.ex} to \eqref{1.eq}-\eqref{1.bic} with
diffusion matrix \eqref{1.SKT}, we infer the following corollary.

\begin{corollary}[Bounded weak solutions to \eqref{1.SKT}]\label{coro.ex}
Let the assumptions of Theorem \ref{thm.ex} hold except that the coefficients
of $A$, defined in \eqref{1.SKT}, are nonnegative and satisfy
$a_{10}>0$, $a_{20}>0$ as well as
\begin{equation}\label{1.a}
  a_{21} = a_{11}, \quad a_{22} = a_{12}, \quad a_{20}-a_{10} = a_{11}-a_{22}\ge 0.
\end{equation}
Furthermore, let $f(u)$ be given by the Lotka-Volterra terms \eqref{1.LV}
satisfying
\begin{equation}\label{1.b}
  b_{10} \le \min\{b_{11},b_{12}\}, \quad b_{20} \le \min\{b_{21},b_{22}\}.
\end{equation}
Then there exists a bounded weak solution $u=(u_1,u_2)$ to \eqref{1.eq}-\eqref{1.bic}
satisfying $u_1$, $u_2\ge 0$, $u_1+u_2\le 1$ in $\Omega\times(0,\infty)$, and
\eqref{1.reg}.
\end{corollary}

\begin{proof}
The corollary follows from Theorem \ref{thm.Jue15} and Theorem \ref{thm.ex}
by specifying the diffusivities according to \eqref{1.SKT}.
The requirement of the symmetry of $H(u)A(u)$ leads to the conditions
$a_{11}=a_{21}$, $a_{22}=a_{12}$, and $a_{20}-a_{10} = a_{11}-a_{22}$, whereas 
\eqref{1.cond} becomes $a_{10}>0$, $a_{20}>0$, and 
$-a_{12}< a_{10}+2\min\{a_{20}-a_{10},0\}$. Taking into account that 
$a_{10}\le a_{20}$, the last condition is equivalent to $-a_{12}< a_{10}$, 
and this inequality holds since $a_{10}$ is positive.
Finally, Hypothesis H3 follows from the inequality 
$g_i(u)=b_{i0}-b_{i1}u_1-b_{i2}u_2\le b_{i0}-\min\{b_{i1},b_{i2}\}(u_1+u_2)\le 0$ 
for $1-\eps<u_1+u_2<1$, where $\eps=\min\{\eps_1,\eps_2\}$ and
$\eps_i=1-b_{i0}/\min\{b_{i1},b_{i2}\}\in(0,1)$.
\end{proof}

%%%%%%%%%%%%%%%%%%%%%%%%%%%%%%%%%%%%%%%%%%%%%%%%%%%%%%%%%%%%%%%%%%%%%%%%%%%%%%%

\section{Proof of Theorem \ref{thm.time}}\label{sec.time}

First, we observe that condition \eqref{1.bij} is a special case of the 
weak competition condition \eqref{1.wcc} which implies that
$U_1>0$ and $U_2>0$. It holds that $U_1+U_2<1$ since otherwise, the
assumption $U_1+U_2\ge 1$ leads in view of condition \eqref{1.bij} to
$$
  0 = f_1(U) = (b_{10}-b_{11}U_1-b_{12}U_2)U_1
	< (b_{10}-b_{12}U_1-b_{12}U_2)U_1 \le b_{10}-b_{12} = 0,
$$
which is a contradiction. Thus, $U\in D$. Furthermore, 
the identity $b_{i0}=b_{i1}U_1+b_{i2}U_2$ allows us to we rewrite $f_i(u)$ as
\begin{equation}\label{3.f1}
  f_i(u) = -u_i\sum_{j=1}^2 b_{ij}(u_j-U_j), \quad i=1,2,
\end{equation}
and the additional condition \eqref{1.bij} leads to
$$
  f_i(u) = -b_{i0}u_iU_3\bigg(\frac{u_i}{U_i}-\frac{u_3}{U_3}\bigg),
	\quad\mbox{where }U_3:=1-U_1-U_2.
$$
For later use, we observe that the entropy density \eqref{a.h} satisfies
\begin{equation}\label{3.fh}
  f(u)\cdot h'(u|U) = -\sum_{i=1}^2 b_{i0}u_iU_3
	\bigg(\frac{u_i}{U_i}-\frac{u_3}{U_3}\bigg)
  \bigg(\log\frac{u_i}{U_i}-\log\frac{u_3}{U_3}\bigg) \le 0
\end{equation}
for all $u\in D$, and we conclude from Theorem \ref{thm.Jue15} that 
$t\mapsto{\mathcal H}[u(t)|U]:=\int_\Omega h(u(x,t)|U)dx$ is nonincreasing.

For the positivity and large-time behavior, we need another functional.
Define
\begin{align*}
  \Phi_\eps(u|U) &= \int_\Omega \phi_\eps(u|U)dx, \quad\mbox{where} \\
  \phi_\eps(u|U) &= \sum_{i=1}^2\bigg(u_i-U_i-(U_i+\eps)
	\log\frac{u_i+\eps}{U_i+\eps}\bigg), \quad u\in D.
\end{align*}
We will show that $\Phi_\eps(u|U)$ is an entropy for \eqref{1.eq}-\eqref{1.bic}.
For this, let $K=\phi_\eps''(u|U)$ be the Hessian of $\phi_\eps$ with
respect to $u$. Because of the $\eps$-regularization, $\phi'_\eps(u|U)$
is an admissible test function for \eqref{1.eq}:
\begin{equation}\label{3.Phi}
  \Phi_\eps(u(t)|U) + \int_0^t\int_\Omega\na u:KA(u)\na udxds
	= \Phi_\eps(u^0|U) + \int_0^t\int_\Omega f(u)\cdot\phi'_\eps(u|U)dxds.
\end{equation}
First, we estimate the last term on the right-hand side. We infer from
\eqref{3.f1} and $\pa_{u_i}\phi_\eps(u|U)=(u_i-U_i)/(u_i+\eps)$ that
\begin{align*}
  \int_0^t\int_\Omega f(u)\cdot\phi'_\eps(u|U)dxds
	&= -\sum_{i,j=1}^2\int_0^t\int_\Omega b_{ij}(u_i-U_i)(u_j-U_j)dxds \\
	&\phantom{xx}{}+ \eps\sum_{i,j=1}^2\int_0^t\int_\Omega\frac{b_{ij}}{u_i+\eps}
	(u_i-U_i)(u_j-U_j)dxds.
\end{align*}
Since $(b_{ij})$ is positive definite and $u_i$ is bounded, there are
constants $c_b>0$ and $C>0$ such that
\begin{equation}\label{3.ineqf}
  \int_0^t\int_\Omega f(u)\cdot\phi'_\eps(u|U)dxds
	\le -c_b\int_0^t\|u-U\|_{L^2(\Omega)}^2ds 
	+ \eps C\sum_{i=1}^2\int_0^t\int_\Omega\frac{dxds}{u_i+\eps}.
\end{equation}
Next, the second term on the left-hand side of \eqref{3.Phi}
is estimated with the help of the following lemma.

\begin{lemma}\label{lem.KA}
There exists $\eps_0>0$ and $c_{KA}>0$ such that for all $0<\eps<\eps_0$
and all $u\in D$, $z=(z_1,z_2)\in\R^2$,
$$
  z^\top KA(u)z \ge c_{KA}\sum_{i=1}^2\frac{z_i^2}{(u_i+\eps)^2}.
$$
\end{lemma}

\begin{proof}
The matrix coefficients of $K$ are explicitly given by 
$K_{ij}=(U_i+\eps)\delta_{ij}/(u_i+\eps)^2$. In order to estimate the
product $z^\top KA(u)z$, we rewrite the coefficients of the diffusion matrix as
$A_{ij}(u)=\sum_{k=1}^3 a_{ij}^{(k)}u_k$, where
$$
  a_{ij}^{(1)} = \alpha_{ij}+\beta_{ij}, \quad
	a_{ij}^{(2)} = \alpha_{ij}+\gamma_{ij}, \quad
	a_{ij}^{(3)} = \alpha_{ij}, \quad i,j=1,2.
$$
Then we need to treat the quadratic form
\begin{align*}
  z^\top & KA(u)z
	= \sum_{k=1}^3 u_k\sum_{i,j=1}^2\frac{U_i+\eps}{(u_i+\eps)^2}a_{ij}^{(k)}z_iz_j \\
	&= \sum_{k=1}^3 u_k\bigg(a_{11}^{(k)}w_1^2
	+ \bigg(a_{12}^{(k)}\sqrt{\frac{U_1+\eps}{U_2+\eps}}\frac{u_2+\eps}{u_1+\eps}
	+ a_{21}^{(k)}\sqrt{\frac{U_2+\eps}{U_1+\eps}}\frac{u_1+\eps}{u_2+\eps}\bigg)
	w_1w_2 + a_{22}^{(k)}w_2^2\bigg),
\end{align*}
where $w_i=z_i\sqrt{U_i+\eps}/(u_i+\eps)$, $i=1,2$.
Because of condition \eqref{1.symm1}, 
$a_{12}^{(2)}=a_{12}^{(3)}=a_{21}^{(1)}=a_{21}^{(3)}=0$, and so, the quadratic form
simplifies to
\begin{align}
  z^\top KA(u)z
	&= \sum_{k=1}^3 u_k\big(a_{11}^{(k)}w_1^2 + a_{22}^{(k)}w_2^2\big) \nonumber \\
	&\phantom{xx}{}
	+ \bigg(a_{12}^{(1)}\sqrt{\frac{U_1+\eps}{U_2+\eps}}\frac{u_1(u_2+\eps)}{u_1+\eps}
	+ a_{21}^{(2)}\sqrt{\frac{U_2+\eps}{U_1+\eps}}\frac{u_2(u_1+\eps)}{u_2+\eps}\bigg)
	w_1w_2 \nonumber \\
	&\ge \sum_{k=1}^3 u_k\big(a_{11}^{(k)}w_1^2 + a_{22}^{(k)}w_2^2\big) 
	\nonumber \\
	&\phantom{xx}{}-\bigg(|a_{12}^{(1)}|\sqrt{\frac{U_1+\eps}{U_2+\eps}}(u_2+\eps)
	+ |a_{21}^{(2)}|\sqrt{\frac{U_2+\eps}{U_1+\eps}}(u_1+\eps)\bigg)|w_1||w_2| 
	\nonumber \\
	&= \sum_{k=1}^3 u_k I_k - \eps\bigg(|a_{12}^{(1)}|\sqrt{\frac{U_1+\eps}{U_2+\eps}}
	+ |a_{21}^{(2)}|\sqrt{\frac{U_2+\eps}{U_1+\eps}}\bigg)|w_1||w_2|, \label{3.aux}
\end{align}
where 
\begin{align*}
  I_1 &= a_{11}^{(1)}w_1^2 + a_{22}^{(1)}w_2^2 
	- |a_{21}^{(2)}|\sqrt{\frac{U_2+\eps}{U_1+\eps}}|w_1||w_2|, \\
	I_2 &= a_{11}^{(2)}w_1^2 + a_{22}^{(2)}w_2^2 
	- |a_{12}^{(1)}|\sqrt{\frac{U_1+\eps}{U_2+\eps}}|w_1||w_2|, \\
	I_2 &= a_{11}^{(3)}w_1^2 + a_{22}^{(3)}w_2^2. 
\end{align*}

Condition \eqref{1.cond} shows that $a_{ii}^{(3)}>0$,
$a_{ii}^{(i)}\ge 0$ for $i=1,2$, and conditions \eqref{1.symm1} and \eqref{1.symm2}
lead to
\begin{align*}
  a_{11}^{(1)}a_{22}^{(1)} - 4\frac{U_2}{U_1}|a_{21}^{(2)}|^2
	&= (\alpha_{11}+\beta_{11})(\alpha_{22}+\beta_{22})
	- 4\frac{U_2}{U_1}(\alpha_{21}+\gamma_{21})^2 \\
	&= (\alpha_{11}+\beta_{11})(\alpha_{11}+\beta_{11}-\beta_{12})
	- 4\frac{U_2}{U_1}\gamma_{21}^2 > 0, \\
	a_{11}^{(2)}a_{22}^{(2)} - 4\frac{U_1}{U_2}|a_{12}^{(1)}|^2
  &= (\alpha_{11}+\gamma_{11})(\alpha_{22}+\gamma_{22})
	- 4\frac{U_1}{U_2}(\alpha_{12}+\beta_{12})^2 \\
	&= (\alpha_{22}+\gamma_{22}-\gamma_{21})(\alpha_{22}+\gamma_{22})
	- 4\frac{U_1}{U_2}\beta_{12}^2 > 0,
\end{align*}
and the positivity of the discriminants
follows from assumption \eqref{1.abc}. As
$\sqrt{(U_i+\eps/(U_j+\eps)}$ is an $\eps$-per\-tur\-ba\-tion of $\sqrt{U_i/U_j}$,
there exist $\delta>0$ and $C>0$ such that, for sufficiently small $\eps>0$,
$$
  I_k \ge 2\delta(w_1^2+w_2^2) - \eps C|w_1||w_2| \ge \delta(w_1^2+w_2^2).
$$
Therefore, still for sufficiently small $\eps>0$, \eqref{3.aux} yields
$$
  z^\top KA(u)z \ge \frac{\delta}{2}(w_1^2+w_2^2)
	= \frac{\delta}{2}\bigg(\frac{U_1+\eps}{(u_1+\eps)^2}z_1^2 
	+ \frac{U_2+\eps}{(u_2+\eps)^2}z_2^2\bigg).
$$
Since $U_1>0$, $U_2>0$, the conclusion follows with $c_{KA}=\delta\min\{U_1,U_2\}/2$.
\end{proof}

We proceed with the proof of Theorem \ref{thm.time}. Employing
Lemma \ref{lem.KA} and estimate \eqref{3.ineqf} in the entropy inequality
\eqref{3.Phi}, it follows that 
\begin{align}
  \Phi_\eps & (u(t)|U) + c_b\int_0^t\|u(s)-U\|_{L^2(\Omega)}ds
	+ c_{KA}\sum_{i=1}^2\int_0^t\int_\Omega\frac{|\na u_i|^2}{(u_i+\eps)^2}dxds 
	\nonumber \\
	&\le \Phi_\eps(u^0|U) + \eps C\sum_{i=1}^2\int_0^t\int_\Omega\frac{dxds}{u_i+\eps}.
	\label{3.aux2}
\end{align}

We wish to pass to the limit $\eps\to 0$. First, we focus on the integral on
the right-hand side of \eqref{3.aux2}:
$$
  \eps\sum_{i=1}^2\int_0^t\int_\Omega\frac{dxds}{u_i+\eps}
	= \sum_{i=1}^2\int_0^t\int_{\{u_i>0\}}\frac{\eps}{u_i+\eps}dxds
	+ \sum_{i=1}^2\int_0^t\mbox{meas}(\{x:u_i(x,s)=0\})ds.
$$
By dominated convergence, we have
$$
   \lim_{\eps\to 0}\sum_{i=1}^2\int_0^t\int_{\{u_i>0\}}\frac{\eps}{u_i+\eps}dxds = 0.
$$
Thus, performing the limit inferior $\eps\to 0$ in \eqref{3.aux2} 
and applying Fatou's lemma, we obtain
\begin{align}
  \Phi & (u(t)|U) + c_b\int_0^t\|u(s)-U\|_{L^2(\Omega)}ds
	+ c_{KA}\sum_{i=1}^2\int_0^t\int_\Omega\frac{|\na u_i|^2}{u_i^2}dxds 
	\nonumber \\
	&\le \Phi(u^0|U) + C\sum_{i=1}^2\int_0^t\mbox{meas}(\{x:u_i(x,s)=0\})ds,
	\label{3.aux3}
\end{align}
where
$$
  \Phi(u|U) = \lim_{\eps\to 0}\Phi_\eps(u|U)
	= \sum_{i=1}^2\int_\Omega U_i\bigg(\frac{u_i}{U_i}-1-\log\frac{u_i}{U_i}\bigg)dx.
$$

If $\mbox{meas}(\{x:u_i(x,t)=0\})>0$ for some $t>0$ and some $i\in\{1,2\}$
then $\Phi(u(t)|U)=+\infty$, which contradicts \eqref{3.aux3}. Thus,
$\mbox{meas}(\{x:u_i(x,t)=0\})=0$ for all $t>0$ and $i=1,2$. This means that
$u_i(x,t)>0$ for a.e.\ $x\in\Omega$, $t>0$, which shows the first property
stated in the theorem.
It follows from \eqref{3.aux3} that $u_i-U_i$, $\na\log u_i
\in L^2(0,\infty;L^2(\Omega))$. In particular, \eqref{3.aux3} implies that
$$
  \int_0^\infty\|u(s)-U\|_{L^2(\Omega)}^2ds < \infty.
$$
Hence, there exists a sequence $t_n\to\infty$ such that $u(t_n)\to U$
strongly in $L^2(\Omega)$ as $n\to\infty$. In view of \eqref{3.fh} and
Theorem \ref{thm.Jue15}, the mapping $t\mapsto{\mathcal H}[u(t)|U]$ 
is nonincreasing. Since $h(u(t_n)|U)\to 0$ as $n\to\infty$, the dominated
convergence theorem and the continuity of $h$ in $\overline D$
(see \eqref{a.h}), we infer that ${\mathcal H}[u(t_n)|U]\to 0$ as $n\to\infty$.
Then the monotonicity of $t\mapsto\mathcal H[u(t)|U]$ implies that 
this convergence holds
for any sequence and ${\mathcal H}[u(t)|U]\to 0$ as $t\to\infty$.
This finishes the proof of Theorem \ref{thm.time}.

%%%%%%%%%%%%%%%%%%%%%%%%%%%%%%%%%%%%%%%%%%%%%%%%%%%%%%%%%%%%%%%%%%%%%%%%%

\section{Proof of Theorem \ref{thm.uni}}\label{sec.uni}

Set $\rho=u_1+u_2$ and $\sigma=u_1-u_2$. A straightforward computation shows
that, thanks to assumptions \eqref{1.symm1}-\eqref{1.symm2} and \eqref{1.condu},
$\rho$ and $\sigma$ solve
\begin{align}
  \pa_t\rho &= \Delta F(\rho), \ t>0, \quad \na\rho\cdot\nu=0\
	\mbox{on }\pa\Omega, \quad \rho(0)=u_1^0+u_2^0\ \mbox{in }\Omega, \label{4.rho} \\
	\pa_t\sigma &= \diver\big(d(\rho)\na\sigma + \sigma\na V(\rho)\big), \ t>0,
	\quad \na\sigma\cdot\nu=0\ \mbox{on }\pa\Omega, \quad 
	\sigma(0)=u_1^0-u_2^0\ \mbox{in }\Omega, \label{4.sigma}
\end{align}
where 
$$
  F(\rho) = \left\{\begin{array}{ll}
	(\alpha_{11}+\beta_{11}\rho)^2/(2\beta_{11}) &\mbox{if }\beta_{11}\neq 0, \\
  \alpha_{11}\rho &\mbox{if } \beta_{11}=0,
	\end{array}\right., \quad
  d(\rho)=\alpha_{11}+(\beta_{11}-\beta_{12})\rho, 
$$
and $V(\rho)=\beta_{12}\rho$.
Observe that, by assumption \eqref{1.cond}, $\alpha_{11}+\beta_{11}-\beta_{12}>0$
and hence, together with $\rho=u_1+u_2\le 1$, it holds that $d(\rho)>0$. Clearly,
the bounded weak solution $u=(u_1,u_2)$ to \eqref{1.eq}-\eqref{1.bic} is unique
if and only if the weak solution $(\rho,\sigma)$ to \eqref{4.rho}-\eqref{4.sigma} is 
unique. First, we prove that \eqref{4.rho} possesses at most one weak solution.
Then the uniqueness result is shown for \eqref{4.sigma}.

The function $F$ is nondecreasing since $\beta_{11}>0$. Thus, by the $H^{-1}$
method, the solution to \eqref{4.rho} is unique. Indeed, if $\rho_1$, $\rho_2$
are two weak solutions to \eqref{4.rho}, their difference satisfies
\begin{equation}\label{4.diff}
  \pa_t(\rho_1-\rho_2)=\Delta(F(\rho_1)-F(\rho_2))\quad\mbox{in }\Omega. 
\end{equation}
Let $w(t)$ be the weak solution to the dual problem
$$
  -\Delta w(t)=\rho_1(t)-\rho_2(t)\quad\mbox{in }\Omega, \quad
	\na w(t)\cdot\nu=0\quad\mbox{on }\pa\Omega, \ t>0.
$$
Then $w\in L^2(0,T;H^1(\Omega))$ and using this function as a test function 
in the weak formulation of \eqref{4.diff}:
\begin{align*}
  0 &= \langle\pa_t(\rho_1-\rho_2),w\rangle 
	+ \int_\Omega\na(F(\rho_1)-F(\rho_2))\cdot\na w dx \\
	&= -\langle\pa_t\Delta w,w\rangle - \int_\Omega(F(\rho_1)-F(\rho_2))\Delta w dx \\
	&= \frac12\frac{d}{dt}\int_\Omega|\na w|^2 dx 
	+ \int_\Omega(F(\rho_1)-F(\rho_2))(\rho_1-\rho_2)dx.
\end{align*}
By the monotonicity of $F$, the last integral is nonnegative, so
$\int_\Omega|\na w(t)|^2 dx$ is nonincreasing in time. But
$\int_\Omega|\na w(0)|^2 dx = 0$, and therefore $w(t)=0$ which implies
that $\rho_1(t)=\rho_2(t)$ for $t>0$.

Next, we consider \eqref{4.sigma} with $\rho$ being a given function.
Let $\sigma_1$, $\sigma_2$ be two weak solutions to \eqref{4.sigma}.
As in \cite{Gaj94}, we introduce the semimetric
$$
  \Xi[\sigma_1,\sigma_2] = S[\sigma_1]+S[\sigma_2]
	-2S\bigg[\frac{\sigma_1+\sigma_2}{2}\bigg], \quad
	S[\sigma]=\int_\Omega \sigma\log\sigma dx.
$$
Because of the strict convexity of $\sigma\mapsto\sigma\log\sigma$, it holds that
$\Xi[\sigma_1,\sigma_2]\ge 0$ and $\Xi[\sigma_1,\sigma_2]=0$ if and only
if $\sigma_1=\sigma_2$. 
Computing the time derivative of $\Xi[\sigma_1,\sigma_2]$, we see that the
drift terms cancel and we end up with
$$
  \frac{d}{dt}\Xi[\sigma_1,\sigma_2]
	= -4\int_\Omega d(\rho)\big(|\na\sqrt{\sigma_1}|^2 + |\na\sqrt{\sigma_2}|^2
	- |\na\sqrt{\sigma_1+\sigma_2}|^2\big)dx.
$$
It was shown in, for instance, \cite[Lemma 10]{ZaJu15} that the integral
is nonnegative (since the Fisher information $\int_\Omega 
d(\rho)|\na\sqrt{\sigma_i}|^2dx$
is subadditive). We infer that $\Xi[\sigma_1(t),\sigma_2(t)]\le
\Xi[\sigma_1(0),\sigma_2(0)]$ for $t>0$. As $\sigma_1$ and $\sigma_2$ have
the same initial data, $\Xi[\sigma_1(0),\sigma_2(0)]=0$ and consequently,
$\Xi[\sigma_1(t),\sigma_2(t)]=0$ for $t>0$. 
Since $\Xi$ is a semimetric, we infer that $\sigma_1(t)=\sigma_2(t)$ for $t>0$,
finishing the proof.

%%%%%%%%%%%%%%%%%%%%%%%%%%%%%%%%%%%%%%%%%%%%%%%%%%%%%%%%%%%%%%%%%%%%%%%%%%%%%%%

\begin{appendix}
\section{Necessary conditions for positive semidefiniteness}

We show that conditions \eqref{1.symm1} and a part of conditions \eqref{1.symm2}
are necessary to apply the boundedness-by-entropy method. 
More precisely, we prove the following result.

\begin{lemma}[Necessary conditions]\label{lem.nece}
We define $h(u)=\sum_{k=1}^3\phi_k(u_k)$, where $u=(u_1,u_2)$, $u_3=1-u_1-u_2$,
and $\phi_k\in C^2(0,1)$ are convex functions satisfying
$\lim_{s\to 0+}\phi_k''(s)=\infty$ for $k=1,2,3$. Let $H=h''(u)\in\R^{2\times 2}$
be the Hessian of $h(u)$ and let $A(u)$ be given by \eqref{1.Aij}. 
If $HA(u)$ is positive semidefinite then
\begin{align}
  & \alpha_{12}=\alpha_{21}=\beta_{21}=\gamma_{12}=0, \label{app.c1} \\
	& \beta_{12} = \alpha_{11}-\alpha_{22}+\beta_{11}-\beta_{22}, \quad
	\gamma_{21} = \alpha_{22}-\alpha_{11}+\gamma_{22}-\gamma_{11}. \label{app.c2}
\end{align}
\end{lemma}

Conditions \eqref{app.c1} correspond to \eqref{1.symm1} needed in
Theorem \ref{thm.ex}. If the coefficients fulfill conditions \eqref{1.symm2} from
Theorem \ref{thm.ex} then also \eqref{app.c2} holds. Functions which satisfy the
assumptions stated above are $\phi(s)=s\log s$, $\phi(s)=s-\log s$, and
$\phi(s)=s^b$ with $b<2$, $b\neq 1$.

\begin{proof}
We write $A(u)=\sum_{k=1}^3 u_kA^{(k)}$, where $A^{(k)}=(a_{ij}^{(k)})_{i,j=1,2}$
are constant matrices and
$$
  a_{ij}^{(1)} = \alpha_{ij}+\beta_{ij}, \quad
	a_{ij}^{(2)} = \alpha_{ij}+\gamma_{ij}, \quad
	a_{ij}^{(3)} = \alpha_{ij}.
$$
Furthermore, we formulate $H=\sum_{k=1}^3\phi''_k(u_k)H^{(k)}$, where
$$
  H^{(1)} = \begin{pmatrix} 1 & 0\\ 0 & 0 \end{pmatrix},\quad
  H^{(2)} = \begin{pmatrix} 0 & 0\\ 0 & 1 \end{pmatrix},\quad
  H^{(3)} = \begin{pmatrix} 1 & 1\\ 1 & 1 \end{pmatrix}.
$$
Then
$$
  HA(u) = \sum_{k,\ell=1}^3\phi''_k(u_k)u_\ell H^{(k)}A^{(\ell)}.
$$

The idea is to study the behavior of $HA(u)$ at the border of the triangle $D$.
We take $u_1=(1-\eps)s$, $u_2=(1-\eps)(1-s)$, and consequently $u_3=\eps$ for
some $\eps$, $s\in(0,1)$ in
$$
  \frac{1}{\phi_3''(u_3)}HA(u)
	= \sum_{\ell=1}^3 u_\ell H^{(3)}A^{(\ell)} + \frac{1}{\phi_3''(u_3)}
	\sum_{\ell=1}^3 u_\ell\big(\phi_1''(u_1)H^{(1)}+\phi_2''(u_2)H^{(2)}\big)A^{(\ell)}
$$
and pass to the limit $\eps\to 0$. By assumption, the left-hand side is a
positive semidefinite matrix. Moreover, since $\phi_3''(u_3)=\phi_3''(\eps)\to
\infty$ as $\eps\to 0$, the last sum on the right-hand side vanishes in the
limit. We deduce that
$$
  \lim_{\eps\to 0}\sum_{\ell=1}^3 u_\ell H^{(3)}A^{(\ell)}
	= H^{(3)}\big(sA^{(1)} + (1-s)A^{(2)}\big)
$$
is positive semidefinite for all $s\in(0,1)$, which implies that
$H^{(3)}A^{(1)}$ and $H^{(3)}A^{(2)}$ are positive semidefinite. By exchanging
the rule of $u_1$, $u_2$, $u_3$, a similar argument shows that
$H^{(i)}A^{(j)}$ is positive semidefinite for all $i=1,2,3$, $j\neq i$.
For any matrix $M=(m_{ij})_{i,j=1,2}$, we have
$$
  H^{(1)}M = \begin{pmatrix} m_{11} & m_{12} \\ 0 & 0 \end{pmatrix}, \
  H^{(2)}M = \begin{pmatrix} 0 & 0 \\ m_{21} & m_{22} \end{pmatrix}, \
	H^{(3)}M = \begin{pmatrix} 
	m_{11}+m_{21} & m_{12}+m_{12} \\ m_{11}+m_{21} & m_{12}+m_{22} 
	\end{pmatrix}.
$$
We verify that $H^{(i)}A^{(j)}$ is positive semidefinite for all $i=1,2,3$, $j\neq i$
if and only if \eqref{app.c1}-\eqref{app.c2} hold.
\end{proof}

\end{appendix}

%%%%%%%%%%%%%%%%%%%%%%%%%%%%%%%%%%%%%%%%%%%%%%%%%%%%%%%%%%%%%%%%%%%%%%%%%%%%%%%

\end{document}